\documentclass[12pt]{amsart}

\usepackage{amssymb}
\usepackage{amsmath}
\usepackage{amsthm}
\usepackage{amsbsy}
\usepackage{bm}
\usepackage{hyperref}
\newcommand{\comment}[1]{}
\usepackage{fullpage}

\newcommand{\clh}{\mathcal{H}}
\newcommand{\clk}{\mathcal{K}}

\theoremstyle{theorem}
    \newtheorem{theorem}{Theorem}
    \newtheorem{lemma}[theorem]{Lemma}

\theoremstyle{definition} 
    \newtheorem{definition}[theorem]{Definition}

    \newtheorem{remark}[theorem]{Remark}
    \newtheorem{example}[theorem]{Example}
    \newtheorem{exercise}[theorem]{Exercise}

\def\<{\langle}
\def\>{\rangle}

\def\bar{\overline}

\newcommand\mnote[1]{} 
\newcommand\be{\begin{equation*}}

\newcommand\ee{\end{equation*}}

\newcommand\ben{\begin{equation}}
\newcommand\een{\end{equation}}
\newcommand\bes{\begin{eqnarray*}}
\newcommand\ees{\end{eqnarray*}}

\newcommand\bex{\begin{exercise}}
\newcommand\eex{\end{exercise}}
\newcommand\beg{\begin{example}}
\newcommand\eeg{\end{example}}
\newcommand\benu{\begin{enumerate}}
\newcommand\eenu{\end{enumerate}}
\newcommand\beit{\begin{itemize}}
\newcommand\eeit{\end{itemize}}
\newcommand\berk{\begin{remark}}
\newcommand\eerk{\end{remark}}
\newcommand\bdefn{\begin{defintion}}
\newcommand\edefn{\end{definition}}
\newcommand\bthm{\begin{theorem}}
\newcommand\ethm{\end{theorem}}
\newcommand\bprf{\begin{proof}}
\newcommand\eprf{\end{proof}}
\newcommand\blem{\begin{lemma}}
\newcommand\elem{\end{lemma}}

\newcommand{\sm}{{\raise0.3ex\hbox{$\scriptstyle \setminus$}}}

\def\CHI{\mathchoice%
{\raise2pt\hbox{$\chi$}}%
{\raise2pt\hbox{$\chi$}}%
{\raise1.3pt\hbox{$\scriptstyle\chi$}}%
{\raise0.8pt\hbox{$\scriptscriptstyle\chi$}}}
\def\smalloplus{\raise1pt\hbox{$\,\scriptstyle \oplus\;$}}

\numberwithin{equation}{section}

\begin{document}

\title[]{$\Gamma$-unitaries, dilation and a natural example}
%
%
%
\author[Bhattacharyya]{ T. Bhattacharyya}

\address{Department of Mathematics,\\
        Indian Institute of Science,\\
        Bangalore 560012, India}

\email{tirtha@math.iisc.ernet.in}
\author[Sau]{H. Sau}
\address{Department of Mathematics,\\
        Indian Institute of Science,\\
        Bangalore 560012, India}

\email{sau10@math.iisc.ernet.in}
\thanks{The authors' research is supported by Department of Science and
Technology, India through the project numbered SR/S4/MS:766/12 and
University Grants Commission Centre for Advanced Studies.}
\date{\today}
\maketitle

\begin{abstract}
 This note constructs an explicit normal boundary dilation for a commuting pair $(S,P)$ of bounded operators with the symmetrized bidisk $$\Gamma=\{(z_1+z_2,z_1z_2):|z_1|,|z_2| \leq 1\}$$ as a spectral set.
 Such explicit dilations had hitherto been constructed only in the unit disk \cite{sfr}, the unit bidisk \cite{ando} and in the tetrablock \cite{sir and me1}. The dilation is minimal and unique under a suitable condition. This paper also contains a natural example of a $\Gamma$-isometry. We compute its associated fundamental operator.
\end{abstract}

\section{Introduction}
This section contains the background and the statements of two main results.

In 1951, von Neumann proved the inequality  $$||f(T)|| \leq \text{sup}\{ |f(z)|: |z| \leq 1 \},$$ where $T$ is a Hilbert space contraction and $f$ is a polynomial. A proof, different from that of von Neumann, emerged when Sz-Nagy proved his dilation theorem: every contraction $T$ can be dilated to a unitary $U$, i.e., if $T$  acts on $\mathcal{H}$, then there is a Hilbert space $\mathcal{K} \supset \mathcal{H}$ and a unitary $U$ on $\mathcal{K}$ such that
$$T^n=P_\mathcal{H}U^n|_\mathcal{H}.$$
\\
Indeed, the proof of von Neumann's inequality then is
$$||f(T)||=||P_\mathcal{H}f(U)|_\mathcal{H}||_\mathcal{H} \leq ||f(U)||_\mathcal{K} \leq \text{sup}\{ |f(z)|: |z| \leq 1 \}$$
because $f(U)$ is a normal operator with $\sigma(f(U))=\{ f(z): z \in \sigma (U) \} \subset \{ f(z): |z|=1 \}$.

It has been a theme of long research whether the converse direction can be pursued. This means that one chooses a compact subset $K$ of the plane or of $\mathbb{C}^d$ for $d >1$, considers a $d$-tuple $\underline{T}=(T_1,T_2,\dots T_d)$ of commuting bounded operators that satisfies
$$ ||f(\underline{T})|| \leq \text{sup}\{ |f(z)|: z \in K \} $$
for all rational functions $f$ with poles off $K$ and tries to see if there is a commuting tuple of bounded normal operators $\underline{N}=(N_1,N_2,\dots N_d)$ with $\sigma$({${\underline{N}}$})$\subset bK$, the distinguished boundary of $K$ such that
$$ f(\underline{T}) = P_\mathcal{H} f(\underline{N})|_\mathcal{H}.$$
The tuple $\underline{N}$ is then called a normal boundary dilation. An explicit construction of such an $\underline{N}$ has succeeded, apart from in the disk \cite{sfr}, only in the bidisk\cite{ando}, although the existence of a dilation is abstractly known for an annulus\cite{agler-ann}.

Here, the symmetrized bidisk $\Gamma$ is polynomially convex. Then, by Oka-Weil theorem, a polynomial dilation is the same as a rational dilation. In other words, $$ T_1^{k_1}\cdots T_d^{k_d}=P_\mathcal{H}N_1^{k_1}\cdots N_d^{k_d}|_\mathcal{H}$$
for $k_1,\dots, k_d \geq 0$.

Consider the class $A(\Gamma)$ of  functions continuous in $\Gamma$ and holomorphic in the interior of $\Gamma$. A boundary of $\Gamma$ (with respect to A($\Gamma$)) is a subset on which every function in $A(\Gamma)$ attains its maximum modulus. It is known that there is a smallest one among such boundaries. This particular smallest one is called the \textit{distinguished boundary} of the symmetrized bidisk and is denoted by $b\Gamma$. It is well-known that $b\Gamma$ is the symmetrization of the torus, i.e., $b\Gamma=\{ (z_1+z_2,z_1z_2): |z_1|=1=|z_2|\}$.
\\
\begin{definition}\label{gamma, iso and uni}
A $\Gamma$-contraction is a commuting pair of bounded operators $(S,P)$ on a Hilbert space $\mathcal{H}$ such that set $\Gamma$ is a spectral set for $(S,P)$, i.e., $$||f(S,P)|| \leq \text{sup}\{ |f(s,p)| : (s,p) \in \Gamma \},$$ for any polynomial $f$ in two variables.
\end{definition}
\begin{definition}\label{gamma, iso and uni1}
A $\Gamma$-unitary $(R,U)$ is a commuting pair of bounded normal operators on a Hilbert space $\mathcal{H}$ such that $\sigma(R,U) \subset b\Gamma$ (this is automatically a $\Gamma$-contraction).
\end{definition}
\begin{definition}\label{gamma, iso and uni2}
A $\Gamma$-isometry is the restriction of a $\Gamma$-unitary to a joint invariant subspace.
\end{definition}

Work of the first author and other co-authors showed in \cite{Sourav da's 1st} that given a $\Gamma$-contraction $(S,P)$, there exists a unique operator $F \in \mathcal{B}(\mathcal{D}_P)$ with numerical radius no greater than 1 which satisfies the fundamental equation
\begin{eqnarray}\label{Maa9}
S-S^*P=D_PFD_P,
\end{eqnarray}
where $D_P=(I-P^*P)^\frac{1}{2}$ is the defect operator of the contraction $P$ and $\mathcal{D}_P=\overline{Ran}D_P$ (The second component of a $\Gamma$-contraction is always a contraction.) This operator $F$ is called the \textit{fundamental operator} of the $\Gamma$-contraction $(S,P)$. The defining criterion of a $\Gamma$-contraction implies that the adjoint pair $(S^*,P^*)$ is also a $\Gamma$-contraction. Consider its fundamental operator $G \in \mathcal{B}(\mathcal{D}_{P^*})$, where $D_{P^*}=(I-PP^*)^\frac{1}{2}$ is the defect operator and $\mathcal{D}_{P^*}=\overline{Ran}D_{P^*}$ is its defect space. This $G$ satisfies
\begin{eqnarray}\label{Maa11}
S^*-SP^*=D_{P^*}GD_{P^*}.
\end{eqnarray}
Our first major result is the construction of $\Gamma$-unitary dilation of a $\Gamma$-contraction explicitly. So let $(S,P)$ be a $\Gamma$-contraction on $\mathcal{H}$. Let $F$  and $G$ be the fundamental operators of the $\Gamma$-contractions $(S,P)$ and $(S^*,P^*)$ respectively. The $\Gamma$-isometry, discovered in \cite{Sourav da's 1st}, that dilates $(S,P)$ is described below. The space is $\tilde{\mathcal{H}}= \mathcal{H} \oplus \mathcal{D}_P \oplus \mathcal{D}_P \oplus \cdots$ which is the same as the minimal isometric dilation space of the contraction $P$. In fact, the second component $V$ of the $\Gamma$-isometric dilation $(T_F,V)$ is the minimal isometric dilation of $P$. So
\begin{eqnarray*}\label{isomatrx}
V=
\left(
\begin{array}{c|cccc}
P & 0 & 0  & 0 &\cdots \\
\hline
D_P & 0 & 0 & 0 &\cdots  \\
0 & I & 0 & 0& \cdots \\
0 & 0 & I & 0 & \cdots\\
\vdots & \vdots & \vdots & \vdots &\ddots
\end{array} \right).
\end{eqnarray*}
The first component $T_F$ is
\begin{eqnarray*}
\left(
\begin{array}{c|cccc}
S & 0 & 0  & 0 &\cdots \\
\hline
F^*D_P & F & 0 & 0 &\cdots  \\
0 & F^* & F & 0& \cdots \\
0 & 0 & F^* & F & \cdots\\
\vdots & \vdots & \vdots & \vdots &\ddots
\end{array} \right).
\end{eqnarray*}
The $\Gamma$-unitary dilation is obtained by extending the $\Gamma$-isometry above. Note that by Definition \ref{gamma, iso and uni2}, every $\Gamma$-isometry is the restriction of a $\Gamma$-unitary to a joint invariant subspace. So the existence of a $\Gamma$-unitary dilation of $(S,P)$ is guaranteed the moment one produces a $\Gamma$-isometric dilation. We construct it below. Just as the $\Gamma$-isometric dilation acts on the space of minimal isometric dilation space of $P$, it turns out that the $\Gamma$-unitary dilation acts on the space of minimal unitary dilation of $P$. For brevity, let us denote $\mathcal{D}_{P^*} \oplus \mathcal{D}_{P^*} \oplus \cdots$ by $l^2(\mathcal{D}_{P^*})$. Note that the isometry $V$ above has a natural unitary extension $U$ on $\tilde{\mathcal{H}} \oplus l^2(\mathcal{D}_{P^*})$. In operator matrix form it is
$$
\left(
                   \begin{array}{cc}
                     V & X^\prime \\
                     0 & Y^\prime \\
                   \end{array}
                 \right)
$$
with respect to the decomposition $\tilde{\mathcal{H}} \oplus l^2(\mathcal{D}_{P^*})$, where $X': l^2(\mathcal{D}_{P^*}) \to \tilde{\mathcal{H}}( = \mathcal{H} \oplus  \mathcal{D}_P \oplus \mathcal{D}_P \oplus \cdots)$ is given by
$$
\left(
                   \begin{array}{cccc}
                     D_{P^*} & 0 & 0 & \cdots\\
                     -P^* & 0 & 0 & \cdots\\
                     0 & 0 & 0 & \cdots\\
                     \vdots & \vdots & \vdots &\ddots
                   \end{array}
                 \right)
$$ and $Y': l^2(\mathcal{D}_{P^*}) \to l^2(\mathcal{D}_{P^*})$ is given by
$$
\left(
\begin{array}{cccc}
0 & I & 0 & \cdots  \\
0 & 0 & I& \cdots \\
0 & 0 & 0 & \cdots\\
\vdots & \vdots & \vdots &\ddots
\end{array} \right)
$$
i.e., the adjoint of the unilateral shift on $l^2(\mathcal{D}_{P^*})$. On the same space, the $\Gamma$-unitary dilation acts. Its first component $R$ is the following extension of $T_F$
$$
\left(
                         \begin{array}{cc}
                           T_F & X \\
                           0 & Y \\
                         \end{array}
                       \right)
$$
with respect to the decomposition $\tilde{\mathcal{H}} \oplus l^2(\mathcal{D}_{P^*})$, where the operators $X: l^2(\mathcal{D}_{P^*}) \to \tilde{\mathcal{H}}$ and $Y: l^2(\mathcal{D}_{P^*}) \to l^2(\mathcal{D}_{P^*})$ are given by
$$
\left(
                   \begin{array}{cccc}
                     D_{P^*}G & 0 & 0 & \cdots\\
                     -P^*G & 0 & 0 & \cdots\\
                     0 & 0 & 0 & \cdots\\
                     \vdots & \vdots & \vdots &\ddots
                   \end{array}
                 \right)
\text{ and }
\left(
\begin{array}{ccccc}
G^* & G & 0 & 0 & \cdots  \\
0 & G^* & G & 0 & \cdots \\
0 & 0 & G^* & G & \cdots\\
\vdots & \vdots & \vdots & \vdots &\ddots
\end{array} \right) \text{ respectively.}
$$
\begin{theorem}\label{first major thm}
The pair $(R, U)$ is a $\Gamma$-unitary dilation of $(S,P)$.
\end{theorem}

Note the similarity of the construction with Sch$\ddot{\text{a}}$ffer's construction in \cite{sfr} of the unitary dilation of a contraction. The crucial inputs are $F$ and $G$ in the construction of $R$. After we completed this work, we came to know that Pal \cite{Sourav da} has independently proved the theorem above.

%
Since the minimal unitary dilation of a contraction is unique, it is natural question to ask about the minimality of the dilation $(R,U)$. Recall that minimality of the unitary dilation $U$ of a contraction $P$ means that the dilation space is
$$
\mathcal{K} = \{U^nh : h \in \mathcal{H} \text{ and } n \in \mathbb{Z}\},
$$
where $U^n$, for a negative $n \in \mathbb{Z}$, means ${U^*}^{|n|}$. The uniqueness means that if $U$ on $\mathcal{K}$ and $\tilde{U}$ on $\mathcal{\tilde{K}}$ are two minimal unitary dilations of $P$ on $\mathcal{H}$, then there is a unitary $W : \mathcal{K} \to \mathcal{\tilde{K}}$ such that $WU=\tilde{U}W$ and $W|_{\mathcal{H}}=I$. We show that the above $\Gamma$-unitary dilation $(R,U)$ of $(S,P)$ is minimal. But uniqueness of minimal dilation usually does not hold good in several variables, e.g., Ando's dilation is known to be not unique, see \cite{Timotin-Li}. However, we show that the dilation constructed above is unique in the sense stated below.
\begin{theorem}[Uniqueness]\label{modified uniqueness reslt}
Let $(S,P)$ be a $\Gamma$-contraction on a Hilbert space $\mathcal{H}$ and $(R,U)$, as defined above, be the $\Gamma$-unitary dilation of $(S,P)$.
\begin{enumerate}
\item[(i)] If $(\tilde{R},U)$ is another $\Gamma$-unitary dilation of $(S,P)$, then $\tilde{R}=R$.
\item[(ii)] If $(\tilde{R},\tilde{U})$ on some Hilbert space $\tilde{\mathcal{K}}$ containing $\mathcal{H}$, is another $\Gamma$-unitary dilation of $(S,P)$ where $\tilde{U}$ is a minimal unitary dilation of $P$, then $(\tilde{R},\tilde{U})$ is unitarily equivalent to $(R,U)$.
\end{enumerate}
\end{theorem}

A great source of examples is function theory, particularly reproducing kernel Hilbert spaces. The Hardy space of the bidisk and the two of its subspaces, the space of symmetric functions of the Hardy space and the space of anti-symmetric functions of the Hardy space have natural $\Gamma$-co-isometries acting on them. We compute their fundamental operators.

\section{Elementary Results On $\Gamma$-Contractions}

This section contains certain preliminary results on $\Gamma$-contractions. Just as
\begin{eqnarray}\label{Maa8}
PD_P=D_{P^*}P
\end{eqnarray}
and its adjoint equation
\begin{eqnarray}\label{Maa10}
D_PP^*=P^*D_{P^*}
\end{eqnarray}
have been known from the time of Sz.- Nagy and Foias, we have a crucial operator equality in the case of a $\Gamma$-contraction $(S,P)$ that relates $S,P$  and the fundamental operator $F$. It is
\begin{eqnarray}\label{Maa7}
D_PS=FD_P+F^*D_PP.
\end{eqnarray}
The adjoint form of this equality involves the $\Gamma$-contraction $(S^*,P^*)$ and its fundamental operator $G$. It is
\begin{eqnarray}\label{rmk1}
D_{P^*}S^*=GD_{P^*}+G^*D_{P^*}P^*.
\end{eqnarray}
The next lemma gives a relation between the fundamental operators of  $\Gamma$-contractions $(S,P)$ and $(S^*,P^*)$. This can be found in \cite{Sourav da's 2nd}.
\begin{lemma}\label{lem:2}
Let $(S,P)$ be a $\Gamma$-contraction and $F$, $G$ are fundamental operators of $(S,P)$ and $(S^*,P^*)$ respectively. Then
\begin{eqnarray}\label{Maa6}
 P^*G=F^*P^*|_{\mathcal{D}_{P^*}}.
\end{eqnarray}
\end{lemma}
\begin{proof}
Note that the L.H.S and the R.H.S of (\ref{Maa6}) are operators from $\mathcal{D}_{P^*}$ to $\mathcal{D}_{P}$.
\begin{eqnarray*}
&&\langle (P^*G-F^*P^*)D_{P^*}h, D_Ph' \rangle \\
&=& \langle P^*GD_{P^*}h, D_Ph' \rangle - \langle F^*P^*D_{P^*}h, D_Ph' \rangle \\
&=& \langle D_PP^*GD_{P^*}h, h' \rangle  - \langle D_PF^*P^*D_{P^*}h, h' \rangle \\
&=& \langle P^*D_{P^*}GD_{P^*}h,h' \rangle - \langle D_PF^*D_PP^*h,h' \rangle  \;\; [\text{ using equation (\ref{Maa10})}]\\
&=& \langle D_{P^*}GD_{P^*}h,P h' \rangle - \langle D_PF^*D_PP^*h,h' \rangle \\
&=& \langle (S^*-SP^*)h, P h' \rangle - \langle (S-S^*P)^*P^*h,h' \rangle \;\;[\text{using equation (\ref{Maa9}) and (\ref{Maa11})}]\\
&=& \langle P^*S^*h,h' \rangle - \langle P^*SP^*h,h' \rangle - \langle S^*P^*h,h'\rangle + \langle P^*SP^*h,h' \rangle= 0.
\end{eqnarray*}
Since $h$ and $h'$ were arbitrary, the proof is complete.
\end{proof}
\begin{remark}\label{rmk2}
If one applies Lemma \ref{lem:2} for the $\Gamma$-contraction $(S^*,P^*)$ in place of $(S,P)$, then the result is  $ PF=G^*P|_{\mathcal{D}_{P}}$.
\end{remark}
The next two lemmas give new relations between the fundamental operators of  $\Gamma$-contractions $(S,P)$ and $(S^*,P^*)$.
\begin{lemma}\label{lem:3}
Let $(S,P)$ be a $\Gamma$-contraction on a Hilbert space $\mathcal{H}$. If $F$ and $G$ are fundamental operators of $(S,P)$ and $(S^*,P^*)$ respectively, then
\begin{eqnarray} \label{Maa}
(SD_P-D_{P^*}GP)|_{\mathcal{D}_P}=D_PF.
\end{eqnarray}
\end{lemma}
\begin{proof}
Note that the L.H.S and the R.H.S of (\ref{Maa}) are operators from $\mathcal{D}_{P}$ to $\mathcal{H}$.
\begin{eqnarray*}
&& (SD_P-D_{P^*}G P)D_P h \\
&=& S(I-P^*P)h - D_{P^*}G P D_P h \\
&=& Sh-SP^*P h- (D_{P^*}GD_{P^*})P h \\
&=& Sh-SP^*P h- S^*Ph+SP^*P h \\
&=& Sh-S^*P h  = D_PFD_P h,   \text{           for all h $\in$ $\mathcal{H}$.}
\end{eqnarray*}
Since $\mathcal{D}_P=\overline{Ran}D_P$ and the operators are bounded, we are done.
\end{proof}
\begin{remark}\label{rmk3}
If one applies Lemma \ref{lem:3} for the $\Gamma$-contraction $(S^*,P^*)$ in place of $(S,P)$, then the result is  $ S^*D_{P^*}-D_{P}FP^*=D_{P^*}G$.
\end{remark}
\begin{lemma}\label{lem:4}
Let $F$ and $G$ be the fundamental operator of $(S,P)$ and $(S^*,P^*)$ respectively. Then
\begin{eqnarray} \label{Maa1}
(F^*D_{P}D_{P^*}-FP^*)|_{\mathcal{D}_{P^*}}=D_{P}D_{P^*}G-P^*G^*.
\end{eqnarray}
\end{lemma}
\begin{proof}
Note that, L.H.S and the R.H.S of (\ref{Maa1}) are operators from $\mathcal{D}_{P^*}$ to $\mathcal{D}_{P}$.
\begin{eqnarray*}
&&(F^*D_{P}D_{P^*}-FP^*)D_{P^*}h\\
&=&F^*D_P(I-PP^*)h-FP^*D_{P^*}h\\
&=& F^*D_P h-F^*D_P PP^*h-FD_P P^*h \;\; [\text{using equation \ref{Maa10}}]\\
&=&F^*D_P h-(F^*D_P P+FD_P)P^*h \\
&=& (F^*D_P-D_PSP^*)h    \;\;\;\;\;\; \;\;\;\;\;\;\;\;\;\; \;\;\;\;\;\;\;\;[\text{ using equation \ref{Maa7}} ]\\
&=& (D_PS^*-P^*G^*D_{P^*})h-D_PSP^*h \;\;[\text{using equation \ref{Maa}}]\\
&=& D_P(S^*-SP^*)h-P^*G^*D_{P^*}h \\
&=& D_P D_{P^*}GD_{P^*}h-P^*G^*D_{P^*}h \\
&=& (D_{P}D_{P^*}G-P^*G^*) D_{P^*}h,  \text{           for all } h \in \mathcal{H}.
\end{eqnarray*}
Since $\mathcal{D}_{P^*}=\overline{Ran}D_{P^*}$ and the operators are bounded, we are done.
\end{proof}

\section{$\Gamma$-Unitary Dilation Of A $\Gamma$-Contraction - Proof of Theorem \ref{first major thm}}

The starting point of the proof of Theorem \ref{first major thm} is the pair $(T_F,V)$ on $\tilde{\mathcal{H}}=\mathcal{H} \oplus l^2(\mathcal{D}_P)$, where
$$ T_F(h \oplus (a_0,a_1,a_2,\dots))=(Sh \oplus (F^*D_Ph+Fa_0,F^*a_0+Fa_1,F^*a_1+Fa_2,\dots))
$$ and
$$
V(h \oplus (a_0,a_1,a_2,\dots))=(Ph \oplus (D_Ph,a_0,a_1,a_2,\dots)).
$$
We know from \cite{Sourav da's 1st} that this pair is a $\Gamma$-isometric dilation for $(S,P)$. So the job reduces to find an explicit $\Gamma$-unitary extension of $(T_F, V)$. For that, it is natural to consider the minimal unitary extension $U$ of $V$ on $\mathcal{K}= \tilde{\mathcal{H}} \oplus l^2({\mathcal{D}_{P^*}})$. The explicit form of $U$ due to Sch$\ddot{\text{a}}$ffer \cite{sfr} is given in Section 1.
Sch$\ddot{\text{a}}$ffer proved that $U$ is the minimal unitary dilation of $P$.


We shall first prove that $(R,U)$ on $\mathcal{K}$, defined in Section 1 is a $\Gamma$-unitary. To be able to do that, we need a tractable characterization of a $\Gamma$-unitary. This can be found in \cite{Sourav da's 1st}. The fourth part of Theorem 2.5 there tells us that a pair of commuting operators $(R, U)$ defined on a Hilbert space $\mathcal{H}$ is a $\Gamma$-unitary if and only if $U$ is unitary and $(R,U)$ is a $\Gamma$-contraction.
So, for our particular $(R,U)$, we shall show that
\begin{enumerate}
\item[(i)] $RU=UR $ \text{ and }
\item[(ii)] $\|f(R,U)\| \leq \|f\|_{\infty, \Gamma}, \text{ for every polynomial $f$ in two variables.}$
\end{enumerate}
To show that
$R=
\begin{pmatrix}
T_F & X \\
0 & Y
\end{pmatrix}$
and
$
U=
\begin{pmatrix}
V & X' \\
0 & Y'
\end{pmatrix}
$ commute, we shall have to show $YY'=Y'Y$ and $XY'+T_FX'=X'Y+VX$.
\begin{eqnarray*}
YY'(a_0,a_1,a_2,\dots) &=&
Y(a_1,a_2,a_3,\dots)
\\
&=& (G^*a_1+Ga_2,G^*a_2+Ga_3,G^*a_3+Ga_4,\dots)
\\
&=&
Y'(G^*a_0+Ga_1,G^*a_1+Ga_2,G^*a_2+Ga_3,\dots)
\\
&=&
Y'Y(a_0,a_1,a_2,\dots).
\end{eqnarray*}
For all $(a_0,a_1,a_2,\dots) \in l^2(\mathcal{D}_{P^*})$ we have
\begin{eqnarray*}
&& (XY'+T_FX')(a_0,a_1,a_2,\dots)
\\
&=& X(a_1,a_2,a_3,\dots) + T_F(D_{P^*}a_0 \oplus (-P^*a_0,0,0,\dots))
\\
&=&
(D_{P^*}Ga_1 \oplus (-P^*Ga_1,0,0,\dots)) + (SD_{P^*}a_0 \oplus ((F^*D_PD_{P^*}-FP^*)a_0,-F^*P^*a_0,0,0,\dots))
\\
&=&
(SD_{P^*}a_0+D_{P^*}Ga_1) \oplus ((F^*D_PD_{P^*}-FP^*)a_0-P^*Ga_1,-F^*P^*a_0,0,0,\dots) \text{ and}
\end{eqnarray*}
\begin{eqnarray*}
&&(X'Y+VX)(a_0,a_1,a_2,\dots)
\\
&=&
X'(G^*a_0+Ga_1,G^*a_1+Ga_2,G^*a_2+Ga_3,\dots)+V(D_{P^*}Ga_0 \oplus (-P^*Ga_0,0,0,\dots))
\\
&=&
((D_{P^*}G^*a_0+D_{P^*}Ga_1)\oplus(-P^*G^*a_0-P^*Ga_1,0,0,\dots)) \\
&& + \; (PD_{P^*}Ga_0 \oplus (D_PD_{P^*}Ga_0,-P^*Ga_0,0,0,\dots))
\\
&=&
((D_{P^*}G^*+PD_{P^*}G)a_0 + D_{P^*}Ga_1) \oplus ((D_PD_{P^*}G-P^*G^*)a_0-P^*Ga_1,-P^*Ga_0,0,0,\dots).
\end{eqnarray*}
The lemmas of the last section will now be useful. By Lemma \ref{lem:4}, Lemma \ref{lem:2} and the equation (\ref{rmk1}), it follows that $XY'+T_FX'=X'Y+VX$.
Thus the proof of commutativity is complete.

We now prove that $R$ is a normal operator. What we first prove is that $R=R^*U$, because this will imply that $R$ is a normal operator. To establish the equality $R=R^*U$, it is equivalent to show the following equalities:

\begin{enumerate}
\item[(a)] $Y=Y^*Y'+X^*X'$,
\item[(b)] $X^*V=0$,
\item[(c)] $X=T_F^*X'$  and
\item[(d)] $T_F=T_F^*V.$
 \end{enumerate}

From the definition of $X$ and $Y$, it is easy to check that
 $$
 X^*(h \oplus (a_0,a_1,a_2,\dots))=(G^*D_{P^*}h-G^*Pa_0,\dots)
 $$
and
$$Y^*(a_0,a_1,a_2,\dots)=(Ga_0,G^*a_0+Ga_1,G^*a_1+Ga_2,\dots).$$
Thus
\begin{eqnarray*}
&&(Y^*Y'+X^*X')(a_0,a_1,a_2,\dots)
\\
&=& Y^*(a_1,a_2,a_3,\dots) +X^*(D_{P^*}a_0 \oplus (-P^*a_0,0,0,\dots))
\\
&=&
(Ga_1,G^*a_1+Ga_2,G^*a_2+Ga_3,\dots) + (G^*(I-PP^*)a_0+G^*PP^*a_0,0,0,\dots)
\\
&=&
(G^*a_0+Ga_1,G^*a_1+Ga_2,G^*a_2+Ga_3,\dots) = Y(a_0,a_1,a_2,\dots),
\end{eqnarray*}
which establishes $(a)$. To prove (b), we use equation (\ref{Maa8}) and see that
\begin{eqnarray*} X^*V(h \oplus (a_0,a_1,a_2,\dots)) & = & X^*(Ph \oplus (D_Ph,a_0,a_1,a_2,\dots)) \\
& = & (G^*D_{P^*}P-G^*PD_P)h,0,0,0,\dots)=0. \end{eqnarray*}
To prove (c), we use Remark \ref{rmk3}, and Lemma \ref{lem:2} to get
\begin{eqnarray*}
T_F^*X'(a_0,a_1,a_2,\dots) & = & T_F^*(D_{P^*}a_0 \oplus (-P^*a_0,0,0,\dots)) \\
& = & (S^*D_{P^*}a_0-D_PFP^*a_0) \oplus (-F^*P^*,0,0,\dots) = X(a_0,a_1,a_2,\dots).\end{eqnarray*}
Since $(T_F,V)$ is a $\Gamma$-isometry, $(d)$ holds, by (\cite{Sourav da's 1st},thm.2.14).

Now we proceed to show that $(R,U)$ satisfies the von-Nuemann inequality. For any polynomial $f$ in two variables we have
$$
f(R,U)=\left(
\begin{array}{cc}
f(T_F,V) & Z_f \\
0 & f(Y,Y') \\
\end{array}
\right),
$$ where $(T_F,V)$ and $(Y,Y')=(M_{G+G^*z},M_z)^*$ are $\Gamma$-contractions and $Z_f$ is an operator depending on $f$. We have  by Lemma 1 of \cite{Perturbation of spectrum} that $\sigma(f(R,U)) \subset \sigma(f(T_F,V)) \cup \sigma(f(Y,Y'))$. Which gives
$$r(f(R,U)) \leq \max \{r(f(T_F,V)) , r(f(Y,Y'))\} \leq \max \{\|f(T_F,V)\| , \|f(Y,Y')\|\} \leq \|f\|_{\infty, \Gamma}.$$
Since $R$ is a normal operator so is $f(R,U)$ and hence we have $r(f(R,U))=\|f(R,U)\|$. This completes the proof of part $(ii)$.
 Hence $(R,U)$ is a $\Gamma$-unitary.

To complete the proof of Theorem \ref{first major thm}, we need to show that $(R,U)$ dilates $(S,P)$. Since both $R$ and $U$ are upper triangular, we have
$
R^mU^n = \begin{pmatrix}
 T_F^mV^n & X_{mn} \\
 0 & Y^m{Y'}^n \\
\end{pmatrix}
$, where $X_{mn}$ is some operator from $l^2(\mathcal{D}_{P^*})$ to $\tilde{\mathcal{H}}$.
Again with respect to the decomposition $\mathcal{H}\oplus l^2({\mathcal{D}_P})$, $T_F$ and $V$ both are lower triangular with $S$ and $P$ in the $(11)$-th entry respectively. So $P_\mathcal{H}{T_F}^mV^n|_\mathcal{H}= S^mP^n$, for every $m,n \geq 0$, where $\mathcal{H}$ is embedded into $\tilde{\mathcal{H}}$ by the map $h \to (h \oplus (0,0,0,\dots))$.
Therefore $P_\mathcal{H}R^mU^n|_\mathcal{H}=S^mP^n$, for every $m,n \geq 0$. This completes the proof of dilation. \qed

\section{minimality and uniqueness}
In this section we prove Theorem \ref{modified uniqueness reslt}. First we remark that the dilation is minimal.

\begin{remark}[Minimality]
Minimality of a commuting normal boundary dilation $\underline{N} = (N_1, N_2, \ldots , N_d)$ on a space $\clk$ of a commuting tuple $(T_1, T_2, \ldots ,T_d)$ of bounded operators on a space $\clh$ means that the space $\clk$ is no bigger than
$$\overline{span} \{ N_1^{k_1} N_2^{k_2} \ldots N_d^{k_d} N_1^{*l_1} N_2^{*l_2} \ldots N_d^{*l_d} h : h \in \clh \mbox{ where } k_i \mbox{ and } l_i \in \mathbb{N} \mbox{ for } i=1,2,\ldots, d \}.$$ Note that the space $\clk$ has to be at least this big. In our construction, the space is just the minimal unitary dilation space of $P$ (which is unique up to unitary equivalence). It is a bit of a surprise that one can find the $\Gamma$-unitary dilation of $(S,P)$ on the same space, while one would have normally expected the dilation space to be bigger. Since no dilation of $(S,P)$ can take place on a space smaller than the minimal unitary dilation space of $P$ (because the dilation has to dilate $P$ as well), our construction of $\Gamma$-unitary dilation is minimal. Indeed, post facto we know from our dilation that
$$\overline{span} \{ R^{m_1} R^{*m_2} U^n h : h \in \clh, m_1, m_2\in \mathbb{N} \mbox{ and } n \in \mathbb{Z} \} = \overline{span} \{ U^n h : h \in \clh \mbox{ and } n \in \mathbb{Z} \}.$$
Note the absence of $R$ on the right side.
\end{remark}

We now prove a weaker version of the uniqueness theorem and then we use it to prove the main result.
\begin{lemma}\label{uniqueness3}
Suppose $(S,P)$ is a $\Gamma$-contraction on a Hilbert space $\mathcal{H}$ and $(R,U)$ is the above $\Gamma$-unitary dilation of $(S,P)$. If $(\tilde{R},U)$ is another $\Gamma$-unitary dilation of $(S,P)$ such that $\tilde{R}$ is an extension of $T_F$, then $\tilde{R}=R$.
\end{lemma}

\begin{proof}
Suppose $(\tilde{R},U)$ is another $\Gamma$-unitary dilation of $(S,P)$, such that $\tilde{R}$ is an extension of $T_F$.
Since $\tilde{R}$ is an extension of $T_F$, $\tilde{R}$ is of the form
$\left(
  \begin{array}{cc}
    T_F & X \\
    0 & Y \\
  \end{array}
\right)$ with respect to the decomposition $\clk = \tilde{\mathcal{H}} \oplus l^2({\mathcal{D}_{P^*}})$. Since
$ U = \left(
        \begin{array}{cc}
          V & X^\prime \\
          0 & Y^\prime \\
        \end{array}
      \right)$
 is unitary and $\tilde{R}U=U\tilde{R}$, we have from easy matrix calculations the followings:
 \begin{eqnarray}\label{Maa4}
 Y'^*Y'+X'^*X'=I, \;\; X'^*V=0,
 \end{eqnarray} and
\begin{eqnarray}\label{Maa3}
\tilde{Y}Y'=Y'\tilde{Y}, \;\; \tilde{X}Y'+T_FX'=X'\tilde{Y}+V\tilde{X}.
\end{eqnarray}
Also since $(\tilde{R},U)$ is a $\Gamma$-unitary, we have $\tilde{R}={\tilde{R}}^*U$ and that gives $\tilde{X}=T_F^*X'$.
\\
So
\begin{eqnarray*}
\tilde{X}(a_0,a_1,a_2,\dots)&=&T_F^*X'(a_0,a_1,a_2,\dots)
\\
&=&
T_F^*(D_{P^*}a_0 \oplus (-P^*a_0,0,0,\dots))
\\
&=&
(S^*D_{P^*}a_0-D_PFP^*a_0) \oplus (-F^*P^*a_0,0,0,\dots)
\\
&=&
(D_{P^*}Ga_0 \oplus (-F^*P^*a_0,0,0,\dots))\;\;\; [\text{by Remark \ref{rmk3}}]
\\
&=&
X(a_0,a_1,a_2,\dots).
\end{eqnarray*}
Now to find $\tilde{Y}$, we proceed as follows:
\\
From second equation of (\ref{Maa3}) we have
\begin{eqnarray*}
&&X'\tilde{Y}+V\tilde{X}=\tilde{X}Y'+T_FX'
\\
&\Rightarrow&X'^*X'\tilde{Y}+X'^*V\tilde{X}=X'^*\tilde{X}Y'+X'^*T_FX' \;\;\;[\text{   multiplying $X'^*$ from left}]
\\
&\Rightarrow& (I-Y'^*Y')\tilde{Y}=X'^*\tilde{X}Y'+X'^*T_FX'  \;\;\; [\text{ using (\ref{Maa4}) }]
\\
&\Rightarrow&
\tilde{Y}^*(I-Y'^*Y')=Y'^*\tilde{X}^*X'+X'^*T_F^*X'. \;\;\;\;\;\;\;\;\;\;\;\;\;\;\;\;\;......(*)
\end{eqnarray*}
Note that $(I-Y'^*Y')$ is the orthogonal projection of $l^2(\mathcal{D}_{P^*})$ onto the first component. Let $x=(a_0,a_1,a_2,\dots)$ be in $l^2(\mathcal{D}_{P^*})$. From ($*$) we get
$$\tilde{Y}^*(a_0,0,0,\dots) =Y'^*\tilde{X}^*X'(a_0,a_1,a_2,\dots)+X'^*T_F^*X'(a_0,a_1,a_2,\dots).$$
Thus
\begin{eqnarray*}
&&Y'^*\tilde{X}^*X'(a_0,a_1,a_2,\dots)+X'^*T_F^*X'(a_0,a_1,a_2,\dots)
\\
&=&
Y'^*\tilde{X}^*(D_{P^*}a_0 \oplus (-P^*a_0,0,0,\dots)) + X'^*T_F^*(D_{P^*}a_0 \oplus (-P^*a_0,0,0,\dots))
\\
&=&
Y'^*((D_{P^*}S-PF^*D_P)D_{P^*}a_0+PFP^*a_0,0,0,\dots) \\
&& + \;  X'^*((S^*D_{P^*}-D_PFP^*)a_0\oplus(-F^*P^*a_0,0,0,\dots))
\\
&=&
(0,(D_{P^*}S-PF^*D_P)D_{P^*}a_0+PFP^*a_0,0,0,\dots) \\
&& + \; (D_{P^*}(S^*D_{P^*}-D_PFP^*)a_0+PF^*P^*a_0,0,0,\dots)
\\
&=&
(D_{P^*}(S^*D_{P^*}-D_PFP^*)a_0+PF^*P^*a_0,(D_{P^*}S-PF^*D_P)D_{P^*}a_0+PFP^*a_0,0,0,\dots).
\end{eqnarray*}

Let us denote the operator $(D_{P^*}(S^*D_{P^*}-D_PFP^*)+PF^*P^*)|_{\mathcal{D}_{P^*}}$ by $C$.
Then we have $\tilde{Y}^*(a_0,0,0,\dots) =(Ca_0,C^*a_0,0,0,\dots)$.
Note that $C$ is an operator from $\mathcal{D}_{P^*}$ to $\mathcal{D}_{P^*}$. We shall show that $C=G$, where $G$ is the fundamental operator of the $\Gamma$-contraction $(S^*,P^*)$. The following computation establishes that.

For $h,h'$ in $\mathcal{H}$, we have
\begin{eqnarray*}
&&\langle CD_{P^*}h,D_{P^*}h' \rangle \\
&=&
\langle (D_{P^*}(S^*D_{P^*}-D_PFP^*)+PF^*P^*)D_{P^*}h,D_{P^*}h' \rangle \\
&=&
\langle D_{P^*}S^*(I-PP^*)h-D_{P^*}(D_PFD_P)P^*h+PF^*P^*D_{P^*}h, D_{P^*}h' \rangle \\
&=&
\langle D_{P^*}S^*h-D_{P^*}S^*PP^*h-D_{P^*}SP^*h+D_{P^*}S^*PP^*h+PF^*P^*D_{P^*}h,D_{P^*}h'\rangle \\
&=&
\langle D_{P^*}S^*h-D_{P^*}S P^*h+PF^*P^*D_{P^*}h, D_{P^*}h' \rangle \\
&=&
\langle D_{P^*}(S^*-SP^*)h+PF^*P^*D_{P^*}h, D_{P^*}h' \rangle\\
&=&\langle D_{P^*}^2GD_{P^*}h+PF^*P^*D_{P^*}h,D_{P^*}h' \rangle \\
&=&
\langle (I-PP^*)GD_{P^*}h,D_{P^*}h' \rangle + \langle F^*P^*D_{P^*}h,P^*D_{P^*}h' \rangle \\
&=&
\langle GD_{P^*}h,D_{P^*}h' \rangle - \langle P^*GD_{P^*}h,P^*D_{P^*}h' \rangle + \langle F^*P^*D_{P^*}h,P^*D_{P^*}h' \rangle \\
&=&
\langle GD_{P^*}h,D_{P^*}h' \rangle -  \langle F^*P^*D_{P^*}h,D_{P^*}h' \rangle + \langle F^*P^*D_{P^*}h,P^*D_{P^*}h' \rangle \;\;\; [ \text{  by Lemma \ref{lem:2}} ]
\\
&=&
\langle GD_{P^*}h,D_{P^*}h' \rangle.
\end{eqnarray*}
Hence $C=G$ and hence for every $a$ in $\mathcal{D}_{P^*}$,
$$\tilde{Y}^*(a,0,0,0,\dots) = (Ga,G^*a,0,0,\dots).$$

We want to compute the action of $\tilde{Y}^*$ on an arbitrary vector. Now using first equation of (\ref{Maa3}), we have
\begin{eqnarray*}
\tilde{Y}^*(\overbrace{0,\dots,0}^\text{$n$ times},a,0,\dots )
&=& \tilde{Y}^*{Y'^*}^n(a,0,0,0,\dots)
\\
&=&
{Y'^*}^n\tilde{Y}^*(a,0,0,0,\dots)\\
& = &
{Y'^*}^n(Ga,G^*a,0,0,\dots) =
(\overbrace{0,\dots,0}^{\text{$n$ times}},Ga,G^*a,0,0,\dots),
\end{eqnarray*}
for every $n \geq 0$. Therefore for an arbitrary element $(a_0,a_1,a_2,\dots) \in l^2(\mathcal{D}_{P^*})$, we have
\begin{eqnarray*}
& & \tilde{Y}^*(a_0,a_1,a_2,\dots) \\
&=&\tilde{Y}^*((a_0,0,0,\dots) + (0,a_1,0,\dots) + (0,0,a_2,\dots) +\cdots )
\\
&=&
(Ga_0,G^*a_0,0,0,\dots) + (0,Ga_1,G^*a_1,0,0,\dots) + (0,0,Ga_2,G^*a_2,0,0,\dots) + \cdots
\\
&=&
(Ga_0,G^*a_0+Ga_1,G^*a_1+Ga_2,\dots).
\end{eqnarray*}
For $(a_0,a_1,a_2,\dots)$ and $(b_0,b_1,b_2,\dots)$ in $l^2(\mathcal{D}_{P^*})$, we have
\begin{eqnarray*}
&&\langle(a_0,a_1,a_2,\dots),\tilde{Y}^*(b_0,b_1,b_2,\dots) \rangle
\\
&=& \langle (a_0,a_1,a_2,\dots), (Gb_0,G^*b_0+Gb_1,G^*b_1+Gb_2,\dots) \rangle
\\
&=&
\langle a_0,G b_0 \rangle + \langle a_1,G^*b_0+G b_1 \rangle + \langle a_2,G^*b_1+G b_2 \rangle + \cdots
\\
&=&
\langle G^*a_0+Ga_1,b_0 \rangle + \langle G^*a_1+Ga_2,b_1 \rangle + \langle G^*a_2+Ga_3, b_2 \rangle + \cdots
\\
&=&
\langle (G^*a_0+Ga_1,G^*a_1+Ga_2,G^*a_2+Ga_3,\dots), (b_0,b_1,b_2,\dots) \rangle.
\end{eqnarray*}
Hence, by definition of the adjoint of an operator, we have
$$\tilde{Y}(a_0,a_1,a_2,\dots)= (G^*a_0+Ga_1,G^*a_1+Ga_2,G^*a_2+Ga_3,\dots)=Y(a_0,a_1,a_2,\dots),$$ for every $(a_0,a_1,a_2,\dots) \in l^2(\mathcal{D}_{P^*}).$
Therefore $\tilde{R}=R$. Hence the proof.
\end{proof}
Note that when we write the operator $U$ with respect to the decomposition $l^2(\mathcal{D}_P) \oplus \mathcal{H} \oplus l^2(\mathcal{D}_{P^*})$ then this is of the form
$$
\begin{pmatrix}
U_1 & U_2 & U_3 \\
0   &  P  & U_4 \\
0   &  0  & U_5 \\
\end{pmatrix},
$$
where $U_1, U_2, U_3, U_4$ and $U_5$ are defined as follows
\begin{eqnarray*}
&&U_1(a_0,a_1,a_2,\dots)=(0,a_0,a_1,\dots),\; U_2(h)=(D_ph,0,0,\dots)
\\
&&U_3(b_0,b_0,b_2,\dots)=(-P^*b_0,0,0,\dots),\; U_4(b_0,b_1,b_2,\dots)=D_{P^*}b_0 \text{ and}
\\
&&U_5(b_0,b_1,b_2,\dots)=(b_1,b_2,b_3,\dots),
\end{eqnarray*}
for all $h \in \mathcal{H}, (a_0,a_1,a_2,\dots) \in l^2(\mathcal{D}_P)$ and $(b_0,b_0,b_2,\dots) \in l^2(\mathcal{D}_{P^*})$. Note that this is the Sch$\ddot{\text{a}}$ffer minimal unitary dilation of the contraction $P$ as in \cite{sfr}(also can be found in \cite{Nagy-Foias}, sec. 5, ch. 1).
\begin{lemma}\label{added lemma}
Let $(R_1,R_2,\dots,R_{n-1},U)$ on $\mathcal{K}$ be a dilation of $(S_1,S_2,\dots,S_{n-1},P)$ on $\mathcal{H}$, where $P$ is a contraction on $\mathcal{H}$ and $U$ on $\mathcal{K}$ is the Sch$\ddot{\text{a}}$ffer minimal unitary dilation of $P$. Then for all $j=1,2,\dots , n-1$, $R_j$ admits a matrix representation of the form
$$
\begin{pmatrix}
* & * & * \\
0 & S_j & * \\
0 & 0 & * \\
\end{pmatrix},
$$
with respect to the decomposition $\mathcal{K}=l^2(\mathcal{D}_P) \oplus \mathcal{H} \oplus l^2(\mathcal{D}_{P^*})$.
\end{lemma}
\begin{proof}
Let $R_j=(R^j_{kl})_{k,l=1}^3$ with respect to $\mathcal{K}=l^2(\mathcal{D}_P) \oplus \mathcal{H} \oplus l^2(\mathcal{D}_{P^*})$ for each $j=1,2,\dots,n-1$. Call $\tilde{\mathcal{H}}=l^2(\mathcal{D}_P) \oplus \mathcal{H}$. Since $U$ is minimal we have $\mathcal{K}=\bigvee_{m=-\infty}^{\infty}U^m\mathcal{H}$ and $\tilde{\mathcal{H}}=\bigvee_{m=0}^{\infty}U^m\mathcal{H}=\bigvee_{m=0}^{\infty}V^m\mathcal{H}$, where $V$ is the minimal isometry dilation of $P$.
Note that
$$
P_{\mathcal{H}}R_j(U^mh)=S_jP^mh= S_jP_{\mathcal{H}}U^mh, \text{ for all $h \in \mathcal{H}, m \in \mathbb{N}$ and $j=1,2,\dots, n-1$.}
$$
Therefore we have $P_{\mathcal{H}}R_j|_{\tilde{\mathcal{H}}}=S_jP_{\mathcal{H}}|_{\tilde{\mathcal{H}}}$ or equivalently $S_j^*=P_{\tilde{\mathcal{H}}}R_j^*|_{\mathcal{H}}$ for all $j=1,2,\dots, n-1$. This shows that $R^j_{21}=0$, for all $j=1,2,\dots,n-1$.
\\
Call $\tilde{\mathcal{N}}=\mathcal{H} \oplus l^2(\mathcal{D}_{P^*})$, then note that $\tilde{\mathcal{N}}=\bigvee_{n=0}^{\infty}{U^*}^n\mathcal{H}$.
We have
$$
P_{\mathcal{H}}R^*_j({U^*}^mh)=S^*_j{P^*}^mh= S^*_jP_{\mathcal{H}}{U^*}^mh, \text{ for all $h \in \mathcal{H}, m \in \mathbb{N}$ and $j=1,2,\dots, n-1$.}
$$
This and a similar argument as above give us $S_j=P_{\tilde{\mathcal{N}}}R_j|_{\mathcal{H}}$. Therefore $R^j_{32}=0$, for all $j=1,2,\dots,n-1$.
\\
So far, we have showed that for each $j=1,2,\dots, n-1$, $R_j$ admits the matrix representation of the form
$$
\begin{pmatrix}
R^j_{11} & R^j_{12} & R^j_{13} \\
0 & S_j & R^j_{23} \\
R^j_{31} & 0 & R^j_{33} \\
\end{pmatrix},
$$
with respect to the decomposition $\mathcal{K}=l^2(\mathcal{D}_P) \oplus \mathcal{H} \oplus l^2(\mathcal{D}_{P^*})$. To show that $R^j_{13}=0$ we proceed as follows:
\\
From the commutativity of $R_j$ with $U$ we get, by an easy matrix calculation
\begin{eqnarray}\label{addition}
R^j_{31}U_1=U_5R^j_{31} \text{ and } R^j_{31}U_2=0,
\end{eqnarray}
(equating the $31^{th}$ entries and $32^{th}$ entries of $R_jU$ and $UR_j$ respectively).
By the definition of $U_2$, we have $RanU_2=Ran(I-U_1U_1^*)$. Therefore $R^j_{31}(I-U_1U_1^*)=0$. Which with the first equation of (\ref{addition}) gives
$R^j_{31} = U_5R^j_{31}U_1^*$. Which gives after n-th iteration $R^j_{31}=U_5^nR^j_{31}{U_1^*}^n$. Now since ${U_1^*}^n \to 0$ as $n \to \infty$, we have that $R^j_{31}=0$ for each $j=1,2,\dots, n-1$. This completes the proof of the lemma.
\end{proof}
Now we are ready to prove Theorem \ref{modified uniqueness reslt}, the main result of this section.\\
\underline{\bf{proof of part (i)}} Since $(\tilde{R},U)$ is a dilation of $(S,P)$, by Lemma \ref{added lemma} we have $\tilde{R}$ of the form
$$
\left(
  \begin{array}{cc}
   T  & \tilde{R}_{12} \\
    0 & \tilde{R}_{22} \\
  \end{array}
  \right)
$$
with respect to the decomposition $\tilde{\mathcal{H}} \oplus l^2(\mathcal{D}_{P^*})$, where $T:\tilde{\mathcal{H}} \to \tilde{\mathcal{H}}$ is of the form
$$
\left(
  \begin{array}{cc}
    T_{11} & T_{12} \\
    0 & S \\
  \end{array}
  \right)
$$
with respect to the decomposition $l^2(\mathcal{D}_P)\oplus \mathcal{H}  $. Since $(T,V)$ on $\tilde{H}$ is the restriction of the $\Gamma$-contraction $(\tilde{R},U)$ to $\tilde{\mathcal{H}}$ and $V$ is an isometry, we have $(T,V)$ a $\Gamma$-isometry. Also note that ${T}^*|_{\mathcal{H}}=S^*$ and $V^*|_{\mathcal{H}}=P^*$. So $(T,V)$ is a $\Gamma$-isometric dilation of $(S,P)$. Also note that $V$ is the Sch$\ddot{\text{a}}$ffer minimal isometric dilation of $P$. Now it follows from the Theorem 4.3(2) of \cite{Sourav da's 1st} that $T=T_F$, where $T_F$ is as in Theorem \ref{first major thm}. Therefore $\tilde{R}$ is an extension of $T_F$. Now the proof follows from Lemma \ref{uniqueness3}.
\\
\underline{\bf{proof of part (ii)}} Since $\tilde{U}$ is a minimal unitary dilation of $P$, there exists a unitary operator $W: \tilde{\mathcal{K}} \to \mathcal{K}$ such that $W\tilde{U}W^*=U$ and $Wh=h$ for all $h \in \mathcal{H}$. This shows that $(W\tilde{R}W^*,W\tilde{U}W^*)$ is another $\Gamma$-unitary dilation of $(S,P)$. But $W\tilde{U}W^*=U$. Hence by part (i) we have $(W\tilde{R}W^*,W\tilde{U}W^*)=(R,U)$. Hence the proof.
\qed
\begin{remark}
As in the case of Ando's dilation of a commuting pair of contractions, a minimal $\Gamma$-unitary dilation of a $\Gamma$-contraction need not be unique (upto unitary equivalence). In this section, we  constructed a particular $\Gamma$-unitary dilation which is the most obvious one because it acts on the minimal unitary dilation space of the contraction $P$. Moreover, if the $\Gamma$-unitary dilation space is no bigger than the minimal unitary dilation space of the contraction $P$, then the $\Gamma$-unitary dilation is unique upto unitary equivalence.
\end{remark}
Let $R$ be an operator on a Hilbert space $\mathcal{K}$. A subspace $\mathcal{H}$ of $\mathcal{K}$ is called semi-invariant for $R$ if $\mathcal{H}=\mathcal{N} \ominus \mathcal{M}$ where both $\mathcal{N}$ and $\mathcal{M}$ are invariant for $R$. It is easy to show that $\mathcal{H}$ is semi-invariant for $R$ if and only if $R$ admits a matrix representation of the form
$$
\begin{pmatrix}
* & * & * \\
0 & S & * \\
0 & 0 & * \\
\end{pmatrix}
$$ on $\mathcal{M} \oplus \mathcal{H} \oplus \mathcal{N}^\perp$, where $S$ is the compression of $R$ to $\mathcal{H}$. Clearly $R$ is a dilation of $S$.
 We end this section with a relevant result which generalizes the Lemma 3.1 in chapter (vi) of \cite{Foias-Frazho}.
\begin{lemma}
Let $(R_1,R_2,\dots,R_n)$ be an n-tuple of commuting operators on a Hilbert space $\mathcal{K}$. Then a subspace $\mathcal{H}$ is a semi-invariant for $(R_1,R_2,\dots,R_n)$ if and only if $(R_1,R_2,\dots,R_n)$ is a dilation of $(S_1,S_2,\dots,S_n)$, where $S_1,S_2,\dots,S_n$ are compressions of $R_1,R_2,\dots,R_n$ respectively to $\mathcal{H}$.
\end{lemma}
\begin{proof}
Assume that $(R_1,R_2,\dots,R_n)$ is a dilation of $(S_1,S_2,\dots,S_n)$. Therefore we have
$$
S_1^{i_1}S_2^{i_2}\cdots S_n^{i_n}=P_{\mathcal{H}}R_1^{i_1}R_2^{i_2}\cdots R_n^{i_n}|_{\mathcal{H}}, \text{ for all $i_1,i_2,\dots,i_n \in \mathbb{N}$}.
$$
Let $\mathcal{N}$ be the following subspace $\mathcal{N}=\bigvee_{i_1,i_2,\dots i_n=0}^{\infty}R_1^{i_1}R_2^{i_2}\cdots R_n^{i_n}H$.
Since $R_i$s commute with each other $\mathcal{N}$ is invariant for each $R_i$. To complete the proof it is sufficient to show that $\mathcal{M}=\mathcal{N} \ominus \mathcal{H}$ is invariant for each $R_i$.
Note that
$$
P_{\mathcal{H}}R_{j}R_1^{i_1}R_2^{i_2}\cdots R_n^{i_n}h=S_jS_1^{i_1}S_2^{i_2}\cdots S_n^{i_n}h=S_jP_{\mathcal{H}}R_1^{i_1}R_2^{i_2}\cdots R_n^{i_n}h,
$$
for all $i_1,i_2,\dots,i_n \in \mathbb{N}$, $j=1,2,\dots,n$ and $h \in \mathbb{H}$. Hence $P_{\mathcal{H}}R_j|_{\mathcal{N}}=S_jP_{\mathcal{H}}|_{\mathcal{N}}$ or equivalently
$S_j^*=P_{\mathcal{N}}R_j^*|_{\mathcal{H}}$.
Using this along with the fact that $\mathcal{M} \subseteq \mathcal{N}$ is orthogonal to $\mathcal{H}$ we have
$$
\langle  R_jm,h \rangle =   \langle m,R_j^*h \rangle = \langle  m,P_{\mathcal{N}}R_j^*h \rangle = \langle  m,S_j^*h \rangle = 0,
$$
for all $m \in \mathcal{M}$, $h \in \mathcal{H}$ and $j=1,2,\dots, n$.
Therefore $R_j\mathcal{M}$ is orthogonal to $\mathcal{H}$, for all $j=1,2,\dots,n$. Hence $\mathcal{M}$ is invariant for each $R_j$. The converse is easy. This completes the proof.
\end{proof}

\section{Examples of fundamental operators}

\subsection{Hardy space of the bidisk} Consider the Hilbert space
$$
H^2(\mathbb{D}^2)=\{ f: \mathbb{D}^2 \to \mathbb{C}: f(z_1,z_2)=\sum_{i=0}^\infty\sum_{j=0}^\infty a_{ij}z_1^iz_2^j \text{ with } \sum_{i=0}^\infty\sum_{j=0}^\infty |a_{ij}|^2 < \infty \}
$$
with the inner product $\langle \sum_{i,j=0}^\infty a_{ij}z_1^iz_2^j, \sum_{i,j=0}^\infty b_{ij}z_1^iz_2^j \rangle = \sum_{i,j=0}^\infty a_{ij}\bar{b_{ij}} $. Note that the operator pair $(M_{z_1+z_2},M_{z_1z_2})$ on $H^2(\mathbb{D}^2)$ is a $\Gamma$-isometry, since it is the restriction of the $\Gamma$-unitary $(M_{z_1+z_2},M_{z_1z_2})$ on $L^2(\mathbb{T}^2)$, where $\mathbb{T}$ denotes the unit circle. For brevity, we call the pair $(M_{z_1+z_2},M_{z_1z_2})$ on $H^2(\mathbb{D}^2)$ by $(S,P)$. In this section, we shall first find the fundamental operator of $(S^*,P^*)$

  Note that every element $f\in H^2(\mathbb{D}^2)$ has the form $f(z_1,z_2)=\sum_{i=0}^\infty\sum_{j=0}^\infty a_{ij}z_1^iz_2^j$ where $a_{ij}\in \mathbb{C},$ for all $i,j \geq 0$. So we can write $f$ in the matrix form
  $$ \left( \left( a_{ij} \right) \right)_{i,j=0}^\infty =
\begin{pmatrix}
a_{00} & a_{01} & a_{02} & \dots \\
a_{10} & a_{11} & a_{12} &\dots  \\
a_{20} & a_{21} & a_{22} & \dots \\
\vdots & \vdots & \vdots & \ddots
\end{pmatrix},
$$
where $(ij)$-th entry in the matrix, denotes the coefficient of $z_1^iz_2^j$ in $f(z_1,z_2)=\sum_{i=0}^\infty\sum_{j=0}^\infty a_{ij}z_1^iz_2^j.$
We shall write the matrix form instead of writing the series. In this notation,
 \begin{equation} S( \; \left( \left( a_{ij} \right) \right)_{i,j=0}^\infty \; ) = \left( a_{(i-1)j} + a_{i(j-1)} \right) \mbox{ and } P( \; \left( \left( a_{ij} \right) \right)_{i,j=0}^\infty \; ) = \left( a_{(i-1)(j-1)} \right) \label{SandP} \end{equation}
with the convention that $a_{ij}$ is zero if either $i$ or $j$ is negative.

\begin{lemma}\label{S^* and P^*} The adjoints of the operators $S$
and $P$ are as follows.
$$S^*\begin{pmatrix}
a_{00} & a_{01} & a_{02} & \dots \\
a_{10} & a_{11} & a_{12} &\dots  \\
a_{20} & a_{21} & a_{22} & \dots \\
\vdots & \vdots & \vdots & \ddots
\end{pmatrix}
=
\begin{pmatrix}
a_{10}+a_{01} & a_{11}+a_{02} & a_{12}+a_{03} & \dots \\
a_{20}+a_{11} & a_{21}+a_{12} & a_{22}+a_{13} &\dots  \\
a_{30}+a_{21} & a_{31}+a_{22} & a_{32}+a_{23} & \dots \\
\vdots & \vdots & \vdots & \ddots
\end{pmatrix}
$$
and
$$P^*\begin{pmatrix}
a_{00} & a_{01} & a_{02} & \dots \\
a_{10} & a_{11} & a_{12} &\dots  \\
a_{20} & a_{21} & a_{22} & \dots \\
\vdots & \vdots & \vdots & \ddots
\end{pmatrix}
=
\begin{pmatrix}
a_{11} & a_{12} & a_{13} & \dots \\
a_{21} & a_{22} & a_{23} &\dots  \\
a_{31} & a_{32} & a_{33} & \dots \\
\vdots & \vdots & \vdots & \ddots
\end{pmatrix}.
$$
\end{lemma}
\begin{proof}
This is a matter of straightforward inner product computation.
\end{proof}
\begin{lemma}\label{defect space}
The defect space of $P^*$ in the matrix form is
$$\mathcal{D}_{P^*}
=
\{ \begin{pmatrix}
a_{00} & a_{01} & a_{02} & \dots \\
a_{10} & 0 & 0 &\dots  \\
a_{20} & 0 & 0 & \dots \\
\vdots & \vdots & \vdots & \ddots
\end{pmatrix}:  |a_{00}|^2 + \sum_{j=1}^\infty|a_{0j}|^2 + \sum_{j=1}^\infty|a_{j0}|^2 <
\infty\}.$$ The defect space in the function form is
$\overline{span}\{ 1, z_1^i,z_2^j: i,j \geq 1 \}$. The defect
operator for $P^*$ is $$D_{P^*}
\begin{pmatrix}
a_{00} & a_{01} & a_{02} & \dots \\
a_{10} & a_{11} & a_{12} &\dots  \\
a_{20} & a_{21} & a_{22} & \dots \\
\vdots & \vdots & \vdots & \ddots
\end{pmatrix}
=
\begin{pmatrix}
a_{00} & a_{01} & a_{02} & \dots \\
a_{10} & 0 & 0 &\dots  \\
a_{20} & 0 & 0 & \dots \\
\vdots & \vdots & \vdots & \ddots
\end{pmatrix}.
$$
\end{lemma}
\begin{proof}
Since $P$ is an isometry, $D_{P^*}$ is a projection onto $Range(P)^{\perp} = H^2(\mathbb{D}^2) \ominus Range(P)$. The rest follows from the formula for $P$ in (\ref{SandP}).
\end{proof}
\begin{definition}\label{fs fund}
Define $B: \mathcal{D}_{P^*} \to \mathcal{D}_{P^*}$ by
\begin{eqnarray}\label{full space}
B\begin{pmatrix}
a_{00} & a_{01} & a_{02} & \dots \\
a_{10} & 0 & 0 &\dots  \\
a_{20} & 0 & 0 & \dots \\
\vdots & \vdots & \vdots & \ddots
\end{pmatrix}
=
\begin{pmatrix}
a_{10}+a_{01} & a_{02} & a_{03} & \dots \\
a_{20} & 0 & 0 &\dots  \\
a_{30} & 0 & 0 & \dots \\
\vdots & \vdots & \vdots & \ddots
\end{pmatrix}
\end{eqnarray} for all $a_{j0}, a_{0j} \in \mathbb{C}, j=0,1,2,\dots \text{ with } |a_{00}|^2 + \sum_{j=1}^\infty|a_{0j}|^2 + \sum_{j=1}^\infty|a_{j0}|^2 < \infty$.
\end{definition}
\begin{lemma}
The operator $B$ as defined in Definition \ref{fs fund} is the fundamental operator of $(S^*,P^*)$.
\end{lemma}
\begin{proof}
To show that $B$ is the fundamental operator of $(S^*,P^*)$, we shall show that $B$ satisfies the fundamental equation $S^*-SP^*=D_{P^*}BD_{P^*}$. Using Lemma \ref{S^* and P^*}, we get
\begin{eqnarray*}
(S^*-SP^*)
\begin{pmatrix}
a_{00} & a_{01} & a_{02} & \dots \\
a_{10} & a_{11} & a_{12} &\dots  \\
a_{20} & a_{21} & a_{22} & \dots \\
\vdots & \vdots & \vdots & \ddots
\end{pmatrix}
=
S^*
\begin{pmatrix}
a_{00} & a_{01} & a_{02} & \dots \\
a_{10} & a_{11} & a_{12} &\dots  \\
a_{20} & a_{21} & a_{22} & \dots \\
\vdots & \vdots & \vdots & \ddots
\end{pmatrix}
-
S
\begin{pmatrix}
a_{11} & a_{12} & a_{13} & \dots \\
a_{21} & a_{22} & a_{23} &\dots  \\
a_{31} & a_{32} & a_{33} & \dots \\
\vdots & \vdots & \vdots & \ddots
\end{pmatrix}
\end{eqnarray*}
\begin{eqnarray*}
&=&
\begin{pmatrix}
a_{10}+a_{01} & a_{11}+a_{02} & a_{12}+a_{03} & \dots \\
a_{20}+a_{11} & a_{21}+a_{12} & a_{22}+a_{13} &\dots  \\
a_{30}+a_{21} & a_{31}+a_{22} & a_{32}+a_{23} & \dots \\
\vdots & \vdots & \vdots & \ddots
\end{pmatrix}
-
\begin{pmatrix}
    0   & a_{11}  & a_{12} & a_{13} &\dots \\
a_{11} & a_{21}+a_{12} & a_{22}+a_{13} &a_{23}+a_{14} &\dots  \\
a_{21} & a_{31}+a_{22} & a_{32}+a_{23} & a_{33}+a_{24} & \dots \\
a_{31} & a_{41}+a_{32} &a_{42}+a_{33} &a_{43}+a_{34}& \cdots \\
\vdots & \vdots & \vdots & \vdots & \ddots
\end{pmatrix}
\\
&=&
\begin{pmatrix}
a_{10}+a_{01} & a_{02} & a_{03} & \dots \\
a_{20} & 0 & 0 &\dots  \\
a_{30} & 0 & 0 & \dots \\
\vdots & \vdots & \vdots & \ddots
\end{pmatrix}.
\end{eqnarray*}
Using Lemma \ref{defect space} and Definition \ref{fs fund}, we get
\begin{eqnarray*}
&&D_{P^*}BD_{P^*}
\begin{pmatrix}
a_{00} & a_{01} & a_{02} & \dots \\
a_{10} & a_{11} & a_{12} &\dots  \\
a_{20} & a_{21} & a_{22} & \dots \\
\vdots & \vdots & \vdots & \ddots
\end{pmatrix}
=
D_{P^*}B
\begin{pmatrix}
a_{00} & a_{01} & a_{02} & \dots \\
a_{10} & 0 & 0 &\dots  \\
a_{20} & 0 & 0 & \dots \\
\vdots & \vdots & \vdots & \ddots
\end{pmatrix}\\
&=&
D_{P^*}
\begin{pmatrix}
a_{10}+a_{01} & a_{02} & a_{03} & \dots \\
a_{20} & 0 & 0 &\dots  \\
a_{30} & 0 & 0 & \dots \\
\vdots & \vdots & \vdots & \ddots
\end{pmatrix}
=
\begin{pmatrix}
a_{10}+a_{01} & a_{02} & a_{03} & \dots \\
a_{20} & 0 & 0 &\dots  \\
a_{30} & 0 & 0 & \dots \\
\vdots & \vdots & \vdots & \ddots
\end{pmatrix}.
\end{eqnarray*}
Hence the proof.
\end{proof}

Now we shall consider two subspaces of the Hilbert space $H^2(\mathbb{D}^2)$. The first one consists of all symmetric functions in $H^2(\mathbb{D}^2)$, i.e.,
$$H_+ = \{f \in H^2(\mathbb{D}^2): f(z_1,z_2)=f(z_2,z_1)\}$$ and the second one consists of all anti-symmetric functions in $H^2(\mathbb{D}^2)$, i.e.,
$$H_- = \{f \in H^2(\mathbb{D}^2): f(z_1,z_2)=-f(z_2,z_1)\}.$$
It can be checked that $H^2(\mathbb{D}^2)=H_+ \oplus H_-$. Since both $H_+$ and $H_-$ are invariant under the pair $(M_{z_1+z_2},M_{z_1z_2})$, the spaces $H_+$ and $H_-$ are reducing for $(M_{z_1+z_2},M_{z_1z_2})$. It can be easily checked from the definition of a $\Gamma$-contraction that a restriction of a $\Gamma$-contraction to an invariant subspace is again a $\Gamma$-contraction. So $(M_{z_1+z_2},M_{z_1z_2})|_{H_+}$ and $(M_{z_1+z_2},M_{z_1z_2})|_{H_-}$ are $\Gamma$-contractions. Since restriction of an isometry to an invariant subspace is again an isometry, $M_{z_1z_2}|_{H_+}$ and $M_{z_1z_2}|_{H_-}$ are isometries. Hence by part$(2)$ of Theorem 2.14 in \cite{Sourav da's 1st}, the pairs $(M_{z_1+z_2},M_{z_1z_2})|_{H_+}$ and $(M_{z_1+z_2},M_{z_1z_2})|_{H_-}$ are $\Gamma$-isometries. For brevity, we shall call the pairs $(M_{z_1+z_2},M_{z_1z_2})|_{H_+}$ and $(M_{z_1+z_2},M_{z_1z_2})|_{H_-}$ by $(S_+,P_+)$ and $(S_-,P_-)$ respectively. We shall find the fundamental operators of $(S_+^*,P_+^*)$ and $(S_-^*,P_-^*)$ respectively.

\subsection{Symmetric case}
Every element $f\in H_+$ has the form $f(z_1,z_2)=\sum_{i=0}^\infty\sum_{j=0}^\infty a_{ij}z_1^iz_2^j$ where $a_{ij}\in \mathbb{C}$ and $a_{ij}=a_{ji}$ for all $i,j \geq 0$. So we can write $f$ in the matrix form
  $$
\begin{pmatrix}
a_{00} & a_{01} & a_{02} & \dots \\
a_{01} & a_{11} & a_{12} &\dots  \\
a_{02} & a_{12} & a_{22} & \dots \\
\vdots & \vdots & \vdots & \ddots
\end{pmatrix}.
$$
In what follows, we shall exhibit the fundamental operator of the $\Gamma$-isometry $(S_+, P_+)$. The results are collected and stated in two lemmas without proof because the proofs are similar to what we did above.
\begin{lemma}\label{S_+^* and P_+^*}
The adjoints of $S_+$ and $P_+$ are as follows.
$$S_+^*
\begin{pmatrix}
a_{00} & a_{01} & a_{02} & \dots \\
a_{01} & a_{11} & a_{12} &\dots  \\
a_{02} & a_{12} & a_{22} & \dots \\
\vdots & \vdots & \vdots & \ddots
\end{pmatrix}
=
\begin{pmatrix}
2a_{01} & a_{11}+a_{02} & a_{12}+a_{03} & \dots \\
a_{11}+a_{02} & 2a_{12} & a_{22}+a_{13} &\dots  \\
a_{12}+a_{03} & a_{22}+a_{13} & 2a_{23} & \dots \\
\vdots & \vdots & \vdots & \ddots
\end{pmatrix}$$
and $$P_+^*
\begin{pmatrix}
a_{00} & a_{01} & a_{02} & \dots \\
a_{01} & a_{11} & a_{12} &\dots  \\
a_{02} & a_{12} & a_{22} & \dots \\
\vdots & \vdots & \vdots & \ddots
\end{pmatrix}
=
\begin{pmatrix}
a_{11} & a_{12} & a_{13} & \dots \\
a_{12} & a_{22} & a_{23} &\dots  \\
a_{13} & a_{23} & a_{33} & \dots \\
\vdots & \vdots & \vdots & \ddots
\end{pmatrix}.
$$

The defect space of $P_+^*$ in the matrix form is
$$\mathcal{D}_{P_+^*}
=
\{ \begin{pmatrix}
a_{00}  & a_{01} & a_{02} & \dots \\
a_{01} & 0 & 0 &\dots  \\
a_{02} & 0 & 0 & \dots \\
\vdots & \vdots & \vdots & \ddots
\end{pmatrix}: a_{0j} \in \mathbb{C}, j\geq0 \text{ with }|a_{00}|^2 + 2\sum_{j=1}^\infty|a_{0j}|^2 <
\infty\}.$$ The defect space in the function form is $
\overline{span}\{ z_1^i+z_2^i: i \geq 0 \}.$ The defect operator
is $D_{P_+^*}
\begin{pmatrix}
a_{00} & a_{01} & a_{02} & \dots \\
a_{01} & a_{11} & a_{12} &\dots  \\
a_{02} & a_{12} & a_{22} & \dots \\
\vdots & \vdots & \vdots & \ddots
\end{pmatrix}
=
 \begin{pmatrix}
a_{00} & a_{01} & a_{02} & \dots \\
a_{01} & 0 & 0 &\dots  \\
a_{02} & 0 & 0 & \dots \\
\vdots & \vdots & \vdots & \ddots
\end{pmatrix}.
$
\end{lemma}

\begin{definition}\label{symm fund}
Define $B_+: \mathcal{D}_{P_+^*} \to \mathcal{D}_{P_+^*}$ by
\begin{eqnarray}\label{full space}
B_+\begin{pmatrix}
a_{00} & a_{01} & a_{02} & \dots \\
a_{01} & 0 & 0 &\dots  \\
a_{02} & 0 & 0 & \dots \\
\vdots & \vdots & \vdots & \ddots
\end{pmatrix}
=
\begin{pmatrix}
2a_{01} & a_{02} & a_{03} & \dots \\
a_{02} & 0 & 0 &\dots  \\
a_{03} & 0 & 0 & \dots \\
\vdots & \vdots & \vdots & \ddots
\end{pmatrix}
\end{eqnarray} for all $a_{0j} \in \mathbb{C}, j \geq 0 \text{ with } |a_{00}|^2 + 2\sum_{j=1}^\infty|a_{0j}|^2 < \infty$.
\end{definition}
\begin{lemma} The operator $B_+$ defined on $\mathcal{D}_{P_+^*}$ is the fundamental operator of $(S_+^*,P_+^*)$.
\end{lemma}

\subsection{Anti-symmetric case}
Every element $f\in H_-$ has the form $f(z_1,z_2)=\sum_{i=0}^\infty\sum_{j=0}^\infty a_{ij}z_1^iz_2^j$ where $a_{ij}\in \mathbb{C}$ and $a_{ij}=-a_{ji}$ for all $i,j \geq 0$. So we can write $f$ in the matrix form
  $$
\begin{pmatrix}
0 & a_{01} & a_{02} & \dots \\
-a_{01} & 0 & a_{12} &\dots  \\
-a_{02} & -a_{12} & 0 & \dots \\
\vdots & \vdots & \vdots & \ddots
\end{pmatrix}.
$$

\begin{lemma}\label{S_-^* and P_-^*} The adjoints of $S_-$ and $P_-$ are
as follows. $$S_-^*
\begin{pmatrix}
0 & a_{01} & a_{02} & \dots \\
-a_{01} & 0 & a_{12} &\dots  \\
-a_{02} & -a_{12} & 0 & \dots \\
\vdots & \vdots & \vdots & \ddots
\end{pmatrix}
=
\begin{pmatrix}
0  & a_{02} & a_{12}+a_{03} & \dots \\
-a_{02} & 0 & a_{13} &\dots  \\
-a_{12}-a_{03} & -a_{13} & 0 & \dots \\
\vdots & \vdots & \vdots & \ddots
\end{pmatrix}$$ and
$$
P_-^*
\begin{pmatrix}
0 & a_{01} & a_{02} & \dots \\
-a_{01} & 0 & a_{12} &\dots  \\
-a_{02} & -a_{12} & 0 & \dots \\
\vdots & \vdots & \vdots & \ddots
\end{pmatrix}
=
\begin{pmatrix}
0 & a_{12} & a_{13} & \dots \\
-a_{12} & 0 & a_{23} &\dots  \\
-a_{13} & -a_{23} & 0 & \dots \\
\vdots & \vdots & \vdots & \ddots
\end{pmatrix}.
$$
 The defect space of $P_-^*$ in
the matrix form is $$\mathcal{D}_{P_-^*} = \{ \begin{pmatrix}
0  & a_{01} & a_{02} & \dots \\
-a_{01} & 0 & 0 &\dots  \\
-a_{02} & 0 & 0 & \dots \\
\vdots & \vdots & \vdots & \ddots
\end{pmatrix}: a_{0j} \in \mathbb{C}, j\geq1 \text{ with } 2\sum_{j=1}^\infty|a_{0j}|^2 <
\infty\}.$$ The defect space in the function form is
$\overline{span}\{ z_1^i-z_2^i: i \geq 1 \}$. The defect operator
is

$$D_{P_-^*}
\begin{pmatrix}
0 & a_{01} & a_{02} & \dots \\
-a_{01} & 0 & a_{12} &\dots  \\
-a_{02} & -a_{12} & 0 & \dots \\
\vdots & \vdots & \vdots & \ddots
\end{pmatrix}
=
 \begin{pmatrix}
0  & a_{01} & a_{02} & \dots \\
-a_{01} & 0 & 0 &\dots  \\
-a_{02} & 0 & 0 & \dots \\
\vdots & \vdots & \vdots & \ddots
\end{pmatrix}.
$$
\end{lemma}

\begin{definition}\label{anti-symm fund}
Define $B_-: \mathcal{D}_{P_-^*} \to \mathcal{D}_{P_-^*}$ by
\begin{eqnarray}\label{full space}
B_-
\begin{pmatrix}
0  & a_{01} & a_{02} & \dots \\
-a_{01} & 0 & 0 &\dots  \\
-a_{02} & 0 & 0 & \dots \\
\vdots & \vdots & \vdots & \ddots
\end{pmatrix}
=
\begin{pmatrix}
0 & a_{02} & a_{03} & \dots \\
-a_{02} & 0 & 0 &\dots  \\
-a_{03} & 0 & 0 & \dots \\
\vdots & \vdots & \vdots & \ddots
\end{pmatrix}
\end{eqnarray} for all $a_{0j} \in \mathbb{C}, j \geq 1 \text{ with }  2\sum_{j=1}^\infty|a_{0j}|^2 < \infty$.
\end{definition}
\begin{lemma}
$B_-$ is the fundamental operator of $(S_-^*,P_-^*)$.
\end{lemma}

\subsection{Explicit unitary equivalence}
The three spaces $H^2(\mathbb D^2), H_+$ and $H_-$ described above
provide us with examples of $\Gamma$-isometries. The respective
operator pairs $(S, P)$, $(S_+, P_+)$ and $S_- , P_-)$ are pure
$\Gamma$-isometries. Agler and Young in Theorem 3.2 of \cite{j.o.t}
proved that any pure $\Gamma$-isometry is unitarily equivalent
to $(M_{\varphi},M_z)$ on $H^2_{\mathcal E}(\mathbb D)$ for some
Hilbert space $\mathcal E$. Moreover, $\varphi$ is linear. It has
been known later in \cite{Sourav da's 2nd} (Theorem 3.1) that
$\mathcal E$ can be taken to be $\mathcal{D}_{P^*}$ and
$\varphi(z)=B^*+Bz$ where $B\in \mathcal{B}(\mathcal{D}_{P^*})$ is
the fundamental operator of the $\Gamma$-coisometry $(S^*,P^*)$.
In the final theorem of this paper, we explicitly find the unitary
operators that implement the unitary equivalence for the pure
$\Gamma$-isometries $(S,P), (S_+,P_+)$ and $(S_-,P_-)$.

\begin{theorem} The three unitary operators are described in three
parts below.

\begin{enumerate}
\item[(a)] The unitary operator $U: H^2(\mathbb{D}^2) \to
    H^2_{\mathcal{D}_{P^*}}(\mathbb{D})$ that satisfies $U^*M_{B^*+zB}U = S$ and $ U^*M_zU = P$ is $Uf(z)= D_{P^*}(I-zP^*)^{-1}f.$

\item[(b)] The unitary operator $U_+: H_+ \to
    H^2_{\mathcal{D}_{P_+^*}}(\mathbb{D})$ that satisfies
    $$U_+^* M_{B_+^*+zB_+}U_+ = S_+ \mbox{ and } U_+^*M_zU_+ = P_+$$
    is simply the restriction of the $U$ above to $H_+$.

\item[(c)] The unitary operator $U_-: H_- \to
    H^2_{\mathcal{D}_{P_-^*}}(\mathbb{D})$ that satisfies
    $$U_-^*M_{B_-^*+zB_-} U_- = S_- \mbox{ and }
    U_-^*M_zU_- = P_-$$ is the restriction of $U$ to $H_-$.
\end{enumerate}
\end{theorem}
\begin{proof}
\begin{itemize}
\item[(a)]
First note that the function $z \mapsto D_{P^*}(I-zP^*)^{-1}f$ is a holomorphic function on $\mathbb{D}$, for every $f \in H^2(\mathbb{D}^2)$. Its Taylor series expansion is
\begin{eqnarray}
&&D_{P^*}(I-zP^*)^{-1}f\nonumber
\\
&=&D_{P^*}(I +zP^*+z^2{P^*}^2+ \cdots)f \nonumber
\\
&=& D_{P^*}f + z D_{P^*}P^*f + z^2D_{P^*}{P^*}^2f + \cdots\nonumber
\\
&=& \label{matrixform} \begin{pmatrix}
a_{00} & a_{01} & a_{02} & \dots \\
a_{10} & 0 & 0 &\dots  \\
a_{20} & 0 & 0 & \dots \\
\vdots & \vdots & \vdots & \ddots
\end{pmatrix}
+
z\begin{pmatrix}
a_{11} & a_{12} & a_{13} & \dots \\
a_{21} & 0 & 0 &\dots  \\
a_{31} & 0 & 0 & \dots \\
\vdots & \vdots & \vdots & \ddots
\end{pmatrix}
+
z^2
\begin{pmatrix}
a_{22} & a_{23} & a_{24} & \dots \\
a_{32} & 0 & 0 &\dots  \\
a_{42} & 0 & 0 & \dots \\
\vdots & \vdots & \vdots & \ddots
\end{pmatrix}
+
\cdots .
\end{eqnarray}
To see that $U$ is an isometry, we do a norm computation.
\begin{eqnarray*}
\lVert Uf \rVert^2_{H^2_{\mathcal{D}_{P^*}}(\mathbb{D})}&=&\lVert D_{P^*}f\rVert^2_{\mathcal{D}_{P^*}} + \lVert D_{P^*}P^*f\rVert^2_{\mathcal{D}_{P^*}} + \lVert D_{P^*}{P^*}^2f\rVert^2_{\mathcal{D}_{P^*}} +\cdots
\\
&=& \lVert f\rVert^2 -\lim_{n \to \infty} \lVert {P^*}^nf \rVert^2 = \lVert f \rVert^2_{H^2(\mathbb{D}^2)}.\;\; [\text{ since $P$ is pure.}]
\end{eqnarray*}
From equation (\ref{matrixform}) it is easy to see that $U$ is onto $H^2_{\mathcal{D}_{P^*}}(\mathbb{D})$.
Therefore $U$ is unitary. We now show that
$U^*M_zU=P$ .
\begin{eqnarray*}\label{U^*M_{z}}
&&U^*M_zU
\begin{pmatrix}
a_{00} & a_{01} & a_{02} & \dots \\
a_{10} & a_{11} & a_{12} &\dots  \\
a_{20} & a_{21} & a_{22} & \dots \\
\vdots & \vdots & \vdots & \ddots
\end{pmatrix}\nonumber
\\
&=&
U^*
\left(
z\begin{pmatrix}
a_{00} & a_{01} & a_{02} & \dots \\
a_{10} & 0 & 0 &\dots  \\
a_{20} & 0 & 0 & \dots \\
\vdots & \vdots & \vdots & \ddots
\end{pmatrix}
+
z^2\begin{pmatrix}
a_{11} & a_{12} & a_{13} & \dots \\
a_{21} & 0 & 0 &\dots  \\
a_{31} & 0 & 0 & \dots \\
\vdots & \vdots & \vdots & \ddots
\end{pmatrix}
+
z^3\begin{pmatrix}
a_{22} & a_{23} & a_{24} & \dots \\
a_{32} & 0 & 0 &\dots  \\
a_{42} & 0 & 0 & \dots \\
\vdots & \vdots & \vdots & \ddots
\end{pmatrix}
+ \cdots
\right)\nonumber
\\
&&\hspace{1cm}=
\begin{pmatrix}
0 & 0 & 0 & 0 & \dots \\
0 & a_{00} & a_{01} & a_{02} & \dots \\
0 & a_{10} & a_{11} & a_{12} &\dots  \\
0 & a_{20} & a_{21} & a_{22} & \dots \\
\vdots & \vdots & \vdots & \vdots& \ddots
\end{pmatrix}
=
P
\begin{pmatrix}
a_{00} & a_{01} & a_{02} & \dots \\
a_{10} & a_{11} & a_{12} &\dots  \\
a_{20} & a_{21} & a_{22} & \dots \\
\vdots & \vdots & \vdots & \ddots
\end{pmatrix}.
\end{eqnarray*}
From the definition of $B$ (Definition \ref{fs fund}), one can easily find that for all $a_{j0}, a_{0j} \in \mathbb{C}, j=0,1,2,\dots$ with $ |a_{00}|^2 + \sum_{j=1}^\infty|a_{0j}|^2 + \sum_{j=1}^\infty|a_{j0}|^2 < \infty$,
\begin{eqnarray}\label{B}
B^*
\begin{pmatrix}
a_{00} & a_{01} & a_{02} & \dots \\
a_{10} & 0 & 0 &\dots  \\
a_{20} & 0 & 0 & \dots \\
\vdots & \vdots & \vdots & \ddots
\end{pmatrix}
=
\begin{pmatrix}
0 & a_{00} & a_{01}  & \dots \\
a_{00} & 0 & 0 &\dots  \\
a_{10} & 0 & 0 & \dots \\
\vdots & \vdots & \ddots
\end{pmatrix}.
\end{eqnarray}
To show that $U^*M_{B^*+zB}U=S$, we first calculate $M_{B^*+zB}U$.
\begin{eqnarray*}
&&
M_{B^*+zB}U
\begin{pmatrix}
a_{00} & a_{01} & a_{02} & \dots \\
a_{10} & a_{11} & a_{12} &\dots  \\
a_{20} & a_{21} & a_{22} & \dots \\
\vdots & \vdots & \vdots & \ddots
\end{pmatrix}
\\
&=&
M_{B^*+Bz}
\left(
\begin{pmatrix}
a_{00} & a_{01} & a_{02} & \dots \\
a_{10} & 0 & 0 &\dots  \\
a_{20} & 0 & 0 & \dots \\
\vdots & \vdots & \vdots & \ddots
\end{pmatrix}
+
z\begin{pmatrix}
a_{11} & a_{12} & a_{13} & \dots \\
a_{21} & 0 & 0 &\dots  \\
a_{31} & 0 & 0 & \dots \\
\vdots & \vdots & \vdots & \ddots
\end{pmatrix}
+
z^2\begin{pmatrix}
a_{22} & a_{23} & a_{24} & \dots \\
a_{32} & 0 & 0 &\dots  \\
a_{42} & 0 & 0 & \dots \\
\vdots & \vdots & \vdots & \ddots
\end{pmatrix}
+
\cdots
\right)
\\
&=&
M_{B^*}
\left(
\begin{pmatrix}
a_{00} & a_{01} & a_{02} & \dots \\
a_{10} & 0 & 0 &\dots  \\
a_{20} & 0 & 0 & \dots \\
\vdots & \vdots & \vdots & \ddots
\end{pmatrix}
+
z\begin{pmatrix}
a_{11} & a_{12} & a_{13} & \dots \\
a_{21} & 0 & 0 &\dots  \\
a_{31} & 0 & 0 & \dots \\
\vdots & \vdots & \vdots & \ddots
\end{pmatrix}
+
z^2\begin{pmatrix}
a_{22} & a_{23} & a_{24} & \dots \\
a_{32} & 0 & 0 &\dots  \\
a_{42} & 0 & 0 & \dots \\
\vdots & \vdots & \vdots & \ddots
\end{pmatrix}
+\cdots
\right)
\\
&+&
M_{B}
\left(
z\begin{pmatrix}
a_{00} & a_{01} & a_{02} & \dots \\
a_{10} & 0 & 0 &\dots  \\
a_{20} & 0 & 0 & \dots \\
\vdots & \vdots & \vdots & \ddots
\end{pmatrix}
+
z^2\begin{pmatrix}
a_{11} & a_{12} & a_{13} & \dots \\
a_{21} & 0 & 0 &\dots  \\
a_{31} & 0 & 0 & \dots \\
\vdots & \vdots & \vdots & \ddots
\end{pmatrix}
+
z^3\begin{pmatrix}
a_{22} & a_{23} & a_{24} & \dots \\
a_{32} & 0 & 0 &\dots  \\
a_{42} & 0 & 0 & \dots \\
\vdots & \vdots & \vdots & \ddots
\end{pmatrix}
+
\cdots
\right)
\\
&=&
\left(
\begin{pmatrix}
0 & a_{00} & a_{01}  & \dots \\
a_{00} & 0 & 0 &\dots  \\
a_{10} & 0 & 0 & \dots \\
\vdots & \vdots & \ddots
\end{pmatrix}
+
z
\begin{pmatrix}
0 & a_{11} & a_{12}  & \dots \\
a_{11} & 0 & 0 &\dots  \\
a_{21} & 0 & 0 & \dots \\
\vdots & \vdots & \ddots
\end{pmatrix}
+
z^2
\begin{pmatrix}
0 & a_{22} & a_{23}  & \dots \\
a_{22} & 0 & 0 &\dots  \\
a_{32} & 0 & 0 & \dots \\
\vdots & \vdots & \ddots
\end{pmatrix}
+
\cdots
\right)
\\
&+&
\left(
 z
\begin{pmatrix}
a_{10}+a_{01} & a_{02} & a_{03} & \dots \\
a_{20} & 0 & 0 &\dots  \\
a_{30} & 0 & 0 & \dots \\
\vdots & \vdots & \vdots & \ddots
\end{pmatrix}
+z^2
\begin{pmatrix}
a_{21}+a_{12} & a_{13} & a_{14} & \dots \\
a_{31} & 0 & 0 &\dots  \\
a_{41} & 0 & 0 & \dots \\
\vdots & \vdots & \vdots & \ddots
\end{pmatrix}\right.
\\
&+&
\left.z^3
\begin{pmatrix}
a_{32}+a_{23} & a_{24} & a_{25} & \dots \\
a_{42} & 0 & 0 &\dots  \\
a_{52} & 0 & 0 & \dots \\
\vdots & \vdots & \vdots & \ddots
\end{pmatrix}
+
\cdots\right)
\end{eqnarray*}
\begin{eqnarray*}
&=&
\begin{pmatrix}
0 & a_{00} & a_{01}  & \dots \\
a_{00} & 0 & 0 &\dots  \\
a_{10} & 0 & 0 & \dots \\
\vdots & \vdots & \ddots
\end{pmatrix}
+
z
\begin{pmatrix}
a_{10}+a_{01} & a_{11}+a_{02} & a_{12}+a_{03} & \dots \\
a_{20}+a_{11} & 0 & 0 & \dots  \\
a_{30}+a_{21} & 0 & 0 & \dots \\
\vdots & \vdots & \vdots & \ddots
\end{pmatrix}
\\
&+&
z^2
\begin{pmatrix}
 a_{21}+a_{12} & a_{22}+a_{13} & a_{23}+a_{14} & \dots  \\
 a_{31}+a_{22} & 0 & 0 & \dots \\
a_{41}+a_{32} & 0 & 0 & \dots \\
\vdots & \vdots & \vdots & \ddots
\end{pmatrix}
+
\cdots.
\end{eqnarray*}
Therefore
\begin{eqnarray}\label{U^*M_{B}}
&&U^*M_{B^*+Bz}U
\begin{pmatrix}
a_{00} & a_{01} & a_{02} & \dots \\
a_{10} & a_{11} & a_{12} &\dots  \\
a_{20} & a_{21} & a_{22} & \dots \\
\vdots & \vdots & \vdots & \ddots
\end{pmatrix}\nonumber
\\
&=&
U^*
\left(
\begin{pmatrix}
0 & a_{00} & a_{01}  & \dots \\
a_{00} & 0 & 0 &\dots  \\
a_{10} & 0 & 0 & \dots \\
\vdots & \vdots & \ddots
\end{pmatrix}
+
z
\begin{pmatrix}
a_{10}+a_{01} & a_{11}+a_{02} & a_{12}+a_{03} & \dots \\
a_{20}+a_{11} & 0 & 0 & \dots  \\
a_{30}+a_{21} & 0 & 0 & \dots \\
\vdots & \vdots & \vdots & \ddots
\end{pmatrix}\right.\nonumber
\\
&+&
\left.z^2
\begin{pmatrix}
 a_{21}+a_{12} & a_{22}+a_{13} & a_{23}+a_{14} & \dots  \\
 a_{31}+a_{22} & 0 & 0 & \dots \\
a_{41}+a_{32} & 0 & 0 & \dots \\
\vdots & \vdots & \vdots & \ddots
\end{pmatrix}
+
\cdots.
\right)\nonumber
\\
&=&
\begin{pmatrix}
0 & a_{00} & a_{01} & a_{02} & \dots \\
a_{00} & a_{10}+a_{01} & a_{11}+a_{02} & a_{12}+a_{03} & \dots \\
a_{10} & a_{20}+a_{11} & a_{21}+a_{12} & a_{22}+a_{13} &\dots  \\
a_{20} & a_{30}+a_{21} & a_{31}+a_{22} & a_{32}+a_{23} & \dots \\
\vdots & \vdots & \vdots & \vdots & \ddots
\end{pmatrix}
=
S\begin{pmatrix}
a_{00} & a_{01} & a_{02} & \dots \\
a_{10} & a_{11} & a_{12} &\dots  \\
a_{20} & a_{21} & a_{22} & \dots \\
\vdots & \vdots & \vdots & \ddots
\end{pmatrix}.
\end{eqnarray}
Therefore $S=U^*M_{B^*+Bz}U$.
\item[(b)] It can be easily checked that $U|_{H_+}$ is onto $H^2_{\mathcal{D}_{P_+^*}}(\mathbb{D})$ and $U|_{H_-}$ is onto $H^2_{\mathcal{D}_{P_-^*}}(\mathbb{D})$. The rest of the argument is as above.
\end{itemize}
\end{proof}


\begin{thebibliography}{99}
\bibitem{agler-ann} J. Agler, {\em{Rational dilation on an annulus}}, Ann. of Math. 121 (1985), 537-563.

\bibitem{j.o.t} J. Agler and N. J. Young, A model theory for $\Gamma$-contractions, J. Operator Theory 49
(2003), 45-60.

\bibitem{ando} T. Ando, {\em{On a Pair of Commuting Contractions}}, Acta Sci. Math. (Szeged)(24)(1963), 88-90.



\bibitem{Sourav da's 1st} T. Bhattacharyya, S. Pal and S. Shyam Roy, {\em{Dilations of $\Gamma$-contractions
by solving operator equations}}, Advances in Mathematics, Volume 230, 2012, 577-606.

\bibitem{Sourav da's 2nd} T. Bhattacharyya and S. Pal, {\em{A functional model for pure $\Gamma$-contractions}}, J. Operator Theory, 71 (2014), 327-339.

\bibitem{sir and me1} T.Bhattacharyya and H. Sau, {\em{ Explicit and unique construction of tetrablock unitary dilation in a certain case}},  Complex Anal. Oper. Theory, DOI 10.1007/s11785-015-0472-9.
%

\bibitem{Foias-Frazho}  C. Foias and A. Frazho, {\em{The Commutant Lifting Approach To Interpolation Problems}}, Operator Theory: Advances and Applications, Vol 44, Birkhauser Verlag, Basel, 1990.

\bibitem{Perturbation of spectrum} D. Hong-Ke and P. Jin, {\em{Perturbation of spectrums of 2 by 2 operator matrices}}, Proceedings
of the American Mathematical Society, 121 (1994), 761-766.

\bibitem{Sourav da} S. Pal, {\em From Stinespring dilation to Sz.-Nagy dilation on the symmetrized bidisk and operator models}, New York J. Math.
20(2014), 645-664.

\bibitem{Pal-Shalit} S. Pal and O. Shalit, {\em{Spectral sets and distinguished varieties in the symmetrized bidisk}}, J. Funct. Anal., 266 (2014), 5779-5800.

\bibitem{sfr}  J.J. Sch$\ddot{\text{a}}$ffer, {\em{On unitary dilations of contractions}}, Proc. Amer. Math. Soc. 6 (1955), 322. MR 16,934c

\bibitem{Timotin-Li}  Li, W. S. and Timotin, D. {\em{The central Ando dilation and related orthogonality properties}} J. Funct. Anal. 154 (1998), 1-16.


\bibitem{Nagy-Foias}  B. Sz.-Nagy, C. Foias, Hari Bercovici and L´aszl´o K´erchy, {\em{Harmonic Analysis of
Operators on Hilbert space}}, Second edition, Revised and enlarged edition, Universitext,
Springer, New York, 2010.

\end{thebibliography}
\end{document}